\documentclass{amsart}%
\usepackage{amssymb}
\usepackage{color}
\usepackage{amsmath}
\usepackage{graphicx}
\usepackage{amsfonts}%
\newcommand{\EE}{\ensuremath{\mathbb{E}}}

\newcommand{\NN}{\ensuremath{\mathbb{N}}}

\newcommand{\PP}{\ensuremath{\mathbb{P}}}

\newcommand{\RR}{\ensuremath{\mathbb{R}}}

\newcommand{\ZZ}{\ensuremath{\mathbb{Z}}}

\renewcommand{\phi}{\varphi}

\newcommand{\eps}{\ensuremath{\epsilon}}

\newcommand{\bB}{\ensuremath{\mathcal{B}}}

\newtheorem{theorem}{Theorem}

\newtheorem{lemma}[theorem]{Lemma}

\newtheorem{remark}[theorem]{Remark}

\newcommand{\ltn}{\ensuremath{\left| \! \left| \! \left|}}
\newcommand{\rtn}{\ensuremath{\right| \! \right| \! \right|}}

\title[Asymptotical stability]
{Asymptotical stability of differential equations driven by H{\"o}lder--continuous paths}

\keywords{differential equations, fractional Brownian motion, pathwise solutions, stability. \\
This work was partially supported by MTM2011-22411, FEDER founding (M.J.
Garrido-Atienza and B. Schmalfu{\ss}), and by ?? (A. Neuenkirch). }

\author{Mar\'{\i}a J. Garrido-Atienza}
\address[Mar\'{\i}a J. Garrido-Atienza]{Dpto. Ecuaciones Diferenciales y An\'alisis Num\'erico\\
Universidad de Sevilla, Apdo. de Correos 1160, 41080-Sevilla,
Spain} \email[Mar\'{\i}a J. Garrido-Atienza]{mgarrido@us.es}

\author{Andreas Neuenkirch}
\address[Andreas Neuenkirch] {Universit\"at Mannheim, Institut f\"ur Mathematik, A5, 6, D-68131, Mannheim, Germany}\email[Andreas Neuenkirch]{neuenkirch@kiwi.math.uni-mannheim.de}

\author{Bj{\"o}rn Schmalfu{\ss }}
\address[Bj{\"o}rn Schmalfu{\ss }]{Institut f\"{u}r Stochastik\\
Friedrich Schiller Universit{\"a}t Jena, Ernst Abbe Platz 2, D-77043\\
Jena,
Germany\\
 }
\email[Bj{\"o}rn Schmalfu{\ss }]{bjoern.schmalfuss@uni-jena.de}

\keywords{differential equations, H\"older--continuous driving signal, fractional Brownian motion, stability. \\
This work was partially supported by MTM2015-63723-P, FEDER funding (M.J.
Garrido-Atienza and B. Schmalfu{\ss}). }

\begin{document}

\begin{abstract}
In this manuscript, we establish asymptotic local exponential  stability of the trivial solution of differential equations driven by H\"older--continuous paths with H\"older exponent greater than $1/2$. This applies in particular to stochastic differential equations driven by fractional Brownian motion with Hurst parameter greater than $1/2$. We motivate the study of local stability by giving a particular example of a scalar equation, where global stability of the trivial solution can be obtained.
\end{abstract}

\maketitle

\section*{\today}

\section{Introduction}
This article is concerned with the study of the stability of $\RR^d$--valued ordinary differential equations driven by H\"older--continuous signals $\omega$ of the form
\begin{align}\label{eq0}
du(t)=F(u(t))dt+G(u(t))d\omega(t),\qquad t\geq 0,
\end{align}
where $u(0)=u_0\in \RR^d$ and $F$, $G$ are smooth functions defined on $\RR^d$. The canonical example for $\omega$ is a sample path of fractional Brownian motion with Hurst index greater than $1/2$. \\

The existence and uniqueness of solutions for ordinary differential equations driven by a H\"older continuous function with H\"older exponent greater than $1/2$ is now well understood, in the context of fractional calculus, see \cite{GMS08}, \cite{NuaRas02} and \cite{Zah98} to name only a few references, and in the context of rough path theory as well, see e.g. \cite{FV10}, \cite{L} and \cite{LQ}.

However, the study of the stability of that kind of equations is in its beginnings, in contrast to the case when the driving noise is a standard Brownian motion, where the existing literature is huge. Let us mention here only a few  pioneering investigations related to stability. Almost sure exponential stability was considered in \cite{K67} for  linear  SDEs  with  Brownian  motion
as integrator using Lyapunov exponents and ergodic theory. In \cite{A84} a.s. exponential stability and uniform boundedness was proved. The multiplicative ergodic theorem of Oseledets allowed the analysis of all exponents for a stochastic flow, leading to a detailed analysis of the dynamics of random systems, see \cite{Arn98}. In \cite{Mao91} the author uses stochastic Lyapunov functions to discuss the stability of stochastic differential equations with semimartingale integrators, making use of the exponential martingale inequality,  obtaining sufficient criteria for a.s. exponential stability and for polynomial stability. In a further book, the same author in \cite{Mao94} gives a consistent account of the theory of stochastic differential equations driven by a nonlinear integrator and their exponential stability at fixed points via Lyapunov function techniques.\\

The fractional Brownian motion $B^H$ is a family of processes indexed by the Hurst parameter $H\in (0,1)$. When $H=1/2$ the fractional Brownian motion is the standard Brownian motion, but when $H\not = 1/2$ the process $B^H$ has properties that differ to those of $B^{1/2}$. In fact, $B^H$ is not a semimartingale nor a Markov process unless $H =1/2$, and therefore the techniques to carry out our study must be different to the ones used to treat the Brownian case. This paper can be seen as a first attempt to analyze the exponential stability of the solution of (\ref{eq0}), and as we will see below, we are able to obtain local stability.  Only in very particular situations in which we can transform the equation into a random ordinary differential equation we can establish global exponential stability, see Section \ref{s} above. Therefore, more efforts are needed in order to cover the global stability and this will be the topic of a forthcoming paper.\\

The long time behavior of solutions has been also carried out by means of the theory of random dynamical systems, by analyzing for instance the existence and inner structure of random attractors and the existence of random fixed points. In \cite{MasSch} the existence of exponentially attracting random fixed points and hence the existence of stationary solutions was shown for linear and semilinear infinite--dimensional stochastic equations with additive fractional Brownian noise with $H>1/2$. In \cite{GKN}, the existence and uniqueness of a stationary solution that pathwise attracts all other solutions was obtained for stochastic differential equations driven by additive fractional Brownian motion with any Hurst parameter. More recently, the papers \cite{GMS08} for a finite--dimensional setting and \cite{GGSch} for an infinite-dimensional one, have been devoted to prove the existence of random attractors for stochastic differential equations with general diffusion coefficients and driven by a fractional Brownian motion with Hurst parameter $H>1/2$, without transformation of the stochastic equation into a random equation, but taking advantage of an $\omega$--wise interpretation of the stochastic integral with respect to the fractional Brownian motion, as we will also do in this article. In both papers, the main step in the proof of the existence of a random attractor consisted of showing the existence of an absorbing set, which was achieved by means of a discrete time version of the Gronwall lemma. In this article our approach also uses a Gronwall-like lemma but we do not require the existence of absorbing sets. Instead, it is based on the study of a sequence $(u^n)_{n\in \NN}$, where each $u^n$ is defined on $[0,1]$ and it is driven by $\theta_n \omega$ (here $\theta$ denotes the Wiener shift flow, see Section \ref{s0}). For each $u^n$ we will derive suitable a priori estimates such we will end up with the local exponential stability of (\ref{eq0}).\\

The remainder of this article is structured as follows: we begin with an example in Section \ref{s}, where we can establish global stability. In Section \ref{s0} we recall some basic facts about fractional Brownian motions and tempered sets, and prove in particular a Gronwall-like lemma, while in Section \ref{s1} we recall the definition and main properties of the Young--integral. Section \ref{s2} concerns the analysis of the existence and uniqueness of solutions, and the definition and the main properties of the localized coefficients  (via a suitable cut-off function) of our equation,  which we will use in our analysis. Finally, Section \ref{s3} is devoted to the study of the local stability of our problem.

\section{Global stability for linear scalar noise}\label{s}

In this section we want to examine the following scalar equation
 \begin{equation}\label{eq-ds}
  du(t)=F(u(t))dt+ \gamma u(t) d\omega(t), \qquad t \geq 0,
\end{equation}
where $F: \RR \rightarrow \RR$ is continuously differentiable with bounded derivative, $\gamma \in \RR$
and $\omega:\RR^+ \rightarrow \mathbb{R}$ is a H\"older--continuous function of order $\beta^{\prime} >1/2$.
Moreover, we assume that we have the splitting
$$  F(x)=-\lambda x+\hat{F}(x), \qquad x \in \mathbb{R}, $$
with $\lambda >0$,
and
\begin{align} \label{ass-F} | \hat{F}(x) |  \leq  \mu |x|, \quad x \in \mathbb{R}, \end{align}
where $0 \leq \mu < \lambda $.

We are going to see that the scalar and linear structure of the noise now allows us to obtain global exponential stability using the Doss-Sussmann transformation.
For this, define first
$$ v(t)=e^{\lambda t}u(t), \qquad t \geq 0.$$
Then the usual change of variable formula, see e.g. Theorem 4.3.1 in \cite{Zah98},  gives
\begin{align} dv(t)= b(t,v(t))dt + \gamma  v(t) d \omega(t), \qquad t \geq 0, \label{damped-transform} \end{align}
where we have set
\begin{align} \label{b-def} b(t,x)=e^{\lambda t}\hat{F}(e^{-\lambda t}x), \qquad t \geq 0, \,\, x \in \mathbb{R}. \end{align}
The results of Doss, see  Theorem 19 in \cite{doss}, state that the solution of  equation \eqref{damped-transform} can be written as
\begin{align} v(t)=h(D(t),\omega(t)), \qquad t \geq 0,  \label{ds-ansatz} \end{align}
where $h: \mathbb{R} \times \mathbb{R} \rightarrow \mathbb{R}$ is the solution of
\begin{align} \label{ds-h} \frac{\partial}{\partial \beta} h(\alpha,\beta)= \gamma  h(\alpha, \beta), \qquad h(\alpha,0)=\alpha, \qquad  \alpha, \beta \in \RR,    \end{align}
i.e. $$  h(\alpha,\beta)=e^{\gamma \beta} \alpha, $$
and $D: \RR^+ \rightarrow  \mathbb{R}$ solves
\begin{align}
 dD(t)&=  e^{  -\gamma \omega(t) }  b \big (t,e^{\gamma \omega(t)} D(t) \big)dt, \qquad t \geq 0, \label{ds-d}
 \\ D(0)&=u(0). \nonumber
\end{align}
The idea behind this representation is to assume that the solution of \eqref{eq-ds} can be written in the form \eqref{ds-ansatz} and to derive necessary and sufficient conditions
for $D$ and $h$, i.e. \eqref{ds-h} and \eqref{ds-d}. In \cite{doss} this approach is introduced for SDEs driven by Brownian motion, but with more general diffusion coefficients, which satisfy a commutativity condition. In our context,
this representation follows by the change of variable formula from Theorem 4.3.1 in \cite{Zah98}.

Using \eqref{b-def}  we obtain
\begin{align*}
dD(t)&=     e^{-\gamma \omega(t) + \lambda t} \hat{F}(e^{\gamma w(t) - \lambda t} D(t) \big)dt, \qquad t \geq 0.
\end{align*}
Now set $r(t)= |D(t)|^2$. Then for $t\geq 0$ we have
\begin{align*} dr(t) &= 2 D(t)  e^{-\gamma \omega(t) + \lambda t} \widehat{F} \big ( e^{\gamma \omega(t) - \lambda t} D(t) \big)dt \leq  2 \mu  | D(t) |^2dt  = 2 \mu r(t)dt,
 \end{align*}
 using assumption \eqref{ass-F}.
Therefore,  Gronwall's Lemma gives
$$   |D(t)|^2= |r(t) | \leq e^{2 \mu t} |u(0)|^{2}, \qquad t \geq 0, $$
and hence  \eqref{ds-ansatz} implies
that
$$ |u(t)| \leq  e^{ \gamma |\omega(t)|} e^{(\mu-\lambda)t} |u(0)|, \qquad t \geq 0.$$
If
\begin{align}
\lim_{t \rightarrow \infty} \frac{|\omega(t)|}{t} =0, \label{w-growth}
\end{align}
then it follows
$$ \lim_{t \rightarrow \infty} e^{ \delta t}|u(t)|=0,$$
for all $ 0\leq \delta < \mu- \lambda$.

Condition \eqref{w-growth} is in particular fulfilled for almost all sample paths of fractional Brownian motion $B^H$, see \eqref{sub-lin-growth}.
Hence we obtain in this case almost sure exponential stability of the zero solution for all rates smaller than $\lambda-\mu$.

So, in the particular situation of  a scalar  equation with a linear multiplicative noise, the above method ensures global exponential stability of the zero solution.
However, this method seems not to be applicable in general when considering a multidimensional driven signal even if the diffusion coefficient is still linear.

As we have said in the Introduction, we want to consider general multidimensional noise perturbations of the type $G(u(\cdot))d\omega(\cdot)$. Our strategy here will be to deal directly with the equation (\ref{eq0}) to obtain local exponential stability.

\section{Preliminaries}\label{s0}

Let us consider the space $C^{\beta^\prime}([0,T];\RR^m)$ of $\beta^\prime$--H{\"o}lder--continuous functions on some interval $[0,T]$ with values in $\RR^m$. The norm in this space is given by
\begin{align*}
  \|\omega\|_{\beta^\prime}&=\|\omega\|_{\infty,0,T}+ \ltn \omega \rtn_{\beta^\prime,0,T}, 
  \end{align*}
  where
  \begin{align*}
  &\|\omega\|_{\infty,0,T}=\sup_{r\in[0,T]}\|\omega(r)\|,\quad \ltn \omega \rtn_{\beta^\prime,0,T}=\sup_{0\le q<r\le T}\frac{\|\omega(r)-\omega(q)\|}{(r-q)^{\beta^\prime}}.
\end{align*}

\medskip

A fractional Brownian motion (fBm) $B^H$ with Hurst parameter $H\in (0,1)$ is a centered Gau{\ss}--process with covariance function
\begin{equation*}
  R(s,t)=\frac12Q(|t|^{2H}+|s|^{2H}-|t-s|^{2H}), \qquad s,t \in \RR,
\end{equation*}
where $Q$ is a non--negative and symmetric matrix in $\RR^m\otimes\RR^m$. For our further purposes it is important to emphasize that the fractional Brownian motion has a version which is $\beta^\prime$--H{\"o}lder--continuous for any $\beta^\prime<H$. Let $C_0(\RR;\RR^m)$ be the set of continuous functions which are zero at zero
equipped with the compact open topology.  Then let $(C_0(\RR;\RR^m),\bB(C_0(\RR;\RR^m)),\PP_H)$ be the
canonical space for  fractional Brownian motion, i.e. $B^H(\omega)=\omega$, where $\PP_H$ denotes the measure of the fBm with Hurst--parameter $H$.  On $C_0(\RR;\RR^m)$ we can introduce the Wiener shift $\theta$ given by
\begin{equation}\label{shift}
  \theta_t\omega(\cdot)=\omega(\cdot+t)-\omega(t),\quad t\in\RR,\quad \omega\in C_0(\RR;\RR^m).
\end{equation}
In particular $\theta_t$ leaves $\PP_H$ invariant.

\medskip

Since the fractional Brownian motion is a Gau{\ss}--process we can show that for $m\in\NN$ and $s,\,t\in \RR$ we have
\begin{equation*}
  \int\|B^H(t)-B^H(s)\|^{2m}d\PP_H(\omega) = c_{m}|t-s|^{2Hm}\quad \text{for all }m\in\NN,
  \end{equation*}
where $c_m= \EE \|B^{1/2}(1)\|^{2m}$.
Applying  Kunita \cite{Kun90} Theorem 1.4.1 we get that
\begin{equation}\label{eq6}
  \EE\|B^H\|_{\beta^{\prime},0,T}^n<\infty
\end{equation}
for any $T>0$, $\beta^{\prime} < H$ and for  $n\in\NN$.
Moreover, since
$$ \mathbb{E} \|B^H\|_{\infty,0,T}^n \leq C_{n,H} T^{nH} $$
for all $n \geq 1$ and all $T>0$, see e.g. Chapter 5.1 in \cite{nualart-book}, the Borel-Cantelli Lemma implies that
\begin{align} \lim_{t \rightarrow \infty} \frac{|B_t^H(\omega)|}{t}=0 \qquad \textrm{for almost all} \quad  \omega \in \Omega. \label{sub-lin-growth} \end{align}

\medskip

The classical Garsia-Rodemich-Rumsey inequality (see \cite{GRR}) states that
$$ \sup_{s,t \in [0,T]}\frac{\|B^{H}_t - B^H_s \|}{|t-s|^{\gamma}}\leq C_{m,\gamma,p}  \left( \int_{0}^T \int_{0}^T \frac{ \|B^{H}_u - B^H_v \|^{2p}}{|u-v|^{2  \gamma p+2}} \, du \,dv \right)^{1/(2p)}
$$
for any $T >0$, $\gamma \in (0,1)$, $p \geq 1$.  Hence \eqref{eq6} implies  by  choosing $\gamma<H$ and $p \geq 1$ sufficiently large that the right hand side of 
the above equation has a finite first moment and hence is $\mathbb{P}_H$-a.s. finite. This is in particular true for the canonical fractional Brownian motion
\begin{equation*}
(C_{0}(\RR;\RR^m),\bB(C_{0}(\RR;\RR^m)),\PP_H)
\end{equation*}
with $B^H(t,\omega)=\omega(t)$.
So, we can conclude that the set $C_0^{\beta^\prime}(\RR;\RR^d)$ of continuous functions which have a finite H{\"o}lder-seminorm on any compact interval and which are zero at zero has $\PP_H$-measure one for $\beta^\prime<H$.
(Alternatively the Kolmogorov Theorem, see Bauer \cite{Bau96} \S39,  could be applied to obtain the same result.)
This set is $\theta$-invariant.


\medskip

A random variable $R\in (0,\infty)$ is called tempered from above if
\begin{equation}\label{eq4b}
\limsup_{t\to \pm\infty}\frac{\log^+R(\theta_t\omega)}{t}=0 \qquad \text{for almost all} \quad \omega \in \Omega.
\end{equation}
Therefore, temperedness from above describes the subexponential growth or decay of a stochastic stationary process
$(t,\omega)\mapsto R(\theta_t\omega)$.
A random variable $R$ is called tempered from below if $R^{-1}$ is tempered from above. In particular, if the random variable $R$ is tempered from below and such that $t \mapsto R(\theta_t\omega)$ is continuous  for all $ \omega \in\Omega$, then for any $\eps>0$ there exists a (random) constant $C_\eps(\omega)>0$ such that
\begin{equation*}
  R(\theta_t\omega)\ge C_\eps(\omega)e^{-\eps|t|} \qquad \text{for almost all} \quad \omega \in \Omega.
\end{equation*}
A sufficient condition for temperedness is that
\begin{equation}\label{eq5}
  \EE\sup_{t\in [0,1]}\log^+R(\theta_t\omega)<\infty.
\end{equation}

By \eqref{eq6} we obtain that $R(\omega)=\|\omega\|_{\beta^\prime,0,1}$ is tempered from above because $\log^+r\le r$ for $r\ge 0$.

Note that the set of all $\omega$ satisfying \eqref{eq4b} is invariant with respect to the flow $\theta$.
\medskip

We need the following simple result.

\begin{lemma}\label{l8} Let
$(R_i)_{i \in \NN}$ and $(v_i)_{i\in\NN}$ be sequences such that
 $R_i\ge C_\eps e^{-\eps i}$ for any $0< \eps< \mu$, $i\in\NN$,
and $v_i\le  v_0e^{-\mu i}$ for any $i \in \NN$, respectively.
Then for sufficiently small $v_0>0$ we have
\begin{equation*}
  v_i\le R_i ,\qquad i \in \NN.
\end{equation*}
\end{lemma}

\medskip

If for instance we assume that a random variable $R>0$ is tempered from below, then we can find a random variable $C_\eps$ such that
$v_i<R(\theta_i\omega)$ holds for $v_0<C_\eps(\omega)$, where $v_i$ satisfies the assumptions of the previous Lemma.

\smallskip
We finish this section with several technical results that will be applied to conclude exponential decay of sequences in further sections.

\begin{lemma}\label{l2}
Let H be a function from $\bar B(0,\rho)\subset \RR^d$ into $\RR^d$, which is continuously differentiable, and zero at zero. Consider the balls $\bar B(0,R), \bar B(0,\hat R)\subset \RR^d$, with $\hat R=\hat R(R)\le\rho$ such that the latter is the
largest centered ball such that
\begin{equation*}
  \bar B(0,{\hat R})\subset H^{-1}(\bar B(0,R)).
\end{equation*}
Then there exists $\kappa \in  (0, \infty)$ such that
\begin{equation*}
  \liminf_{R\to 0}\frac{\hat R(R)}{R}\ge \kappa.
\end{equation*}
\end{lemma}
\begin{proof}
For every sufficiently small $R$ there is an element $x_R$ in $\partial \bar B(0,\hat R)$ such that
\begin{equation*}
  \|H(x_R)\|=R.
\end{equation*}
To see the existence of  such an element consider the continuous function
\begin{equation*}
  f_H:\bar B(0,\rho)\to\RR^+,\quad f_H(u)=\|H(u)\|
\end{equation*}
that verifies $f_H(0)=0$, $R \leq \max_{x \in \bar B(0,\rho)} f_H(x)$. Then we have
\begin{align*}
  f_H^{-1}(\{R\})& =f_H^{-1}([0,R]\cap [R,\infty))=f_H^{-1}([0,R])\cap f_H^{-1}([R,\infty)) \\ & \supset f_H^{-1}([0,R])\cap\overline{f^{-1}_H([0,R])^c}
\end{align*}
and hence all arguments from the boundary of $f_H^{-1}([0,R])$ have the value $R$ with respect to $f_H$.  Note that the largest radius $\hat R$ is given by the infimum of the distances between the boundary $\partial H^{-1}(\bar B(0,R))=\partial f_H^{-1}([0,R])$ and zero. Since this boundary  is a compact set in $\bar B(0,\rho)$ and the mapping $x\rightarrow \|x\|$ is continuous, we have an element $x_R\in \partial f_H^{-1}([0,R])$ such that
\begin{equation*}
\hat R=\inf_{x\in \partial f_H^{-1}([0,R])} \|x\|= \|x_R\|.
\end{equation*}

Finally, applying the mean value theorem it follows
\begin{equation*}
  \frac{\|x_R\|}{R}=\frac{\|x_R\|}{\|H(x_R)\|}\ge \frac{1}{\sup_{\xi\in \bar B(0,\rho)}\|DH(\xi)\|}>0.
\end{equation*}
\end{proof}

For completeness, we state the following measurability result.

\begin{lemma}\label{l7}
Let  $H:[0, \infty) \rightarrow [0, \infty)$ be continuous and non-decreasing. Then the function
$$ J: [0, \infty) \rightarrow [0, \infty), \qquad J(x)= \max \{ r \geq 0: \, H(r) \leq x \}$$
is Borel--measurable.
\end{lemma}
\begin{proof}
Let $\alpha \geq 0$ and consider the set
$$ M(\alpha)= \{ x \in [0, \infty): \, J(x) < \alpha \}.$$
Then we have to check that $M(\alpha)$ belongs to $\mathcal{B}([0,\infty))$. Clearly,  $M(0)= \emptyset$, so assume $\alpha >0$.
By definition and since $H$ is non-decreasing, $J(x)< \alpha$ implies $H(\alpha) >x$. Vice versa $H(\alpha) >x$ implies $J(x) < \alpha$.
Hence
$$ M(\alpha) = \{ x \in [0, \infty): \, x < H(\alpha)  \} =[0, H(\alpha)).$$
\end{proof}

Now we investigate the H\"older-norm of a finite--dimensional semigroup $e^{A \cdot}$ generated by an operator $A$, whose estimates will be used below. The main assumption is that the spectrum of $A$ has a negative real part.
\begin{lemma}\label{l1}
Let  $e^{A \cdot}$ be the fundamental solution to
\begin{equation*}
  u^\prime=Au.
\end{equation*}
Let ${\rm Re}\,\sigma(A)<-\lambda<0$. Then there exists an $M\ge 1$ such that
\begin{equation*}
  \|e^{A\,t}\|\le Me^{-\lambda t}.
\end{equation*}
In addition, we have for $0\le s<t$
\begin{align}\label{semi}
\|e^{At}-e^{As}\|\leq M \|A\| (t-s) e^{-\lambda s},\qquad  \|e^{A(t-s)}-{\rm id} \|\leq M\|A\|(t-s),
\end{align}
where $\|A\|$ is the Euclidean norm of $A$.
\end{lemma}
The proof follows easily by the mean value theorem and Amann \cite{Ama90} Chapter 13.
As a consequence, for $0<s<t$ we have
\begin{align}\label{sem}
\begin{split}
\ltn e^{A(t-\cdot)}\rtn_{\beta,0,t}&=\sup_{0\leq r_1<r_2<t}\frac{\|e^{A(t-r_2)}-e^{A(t-r_1)}\|}{(r_2-r_1)^\beta}\le M\|A\| t^{1-\beta}
\end{split}
\end{align}
and
\begin{align}\label{sem1}
\begin{split}
\ltn e^{A(t-\cdot)}-e^{A(s-\cdot)}\rtn_{\beta,0,s} &=\sup_{0\leq r_1<r_2<s}\frac{\|e^{A(t-r_2)}-e^{A(s-r_2)}-(e^{A(t-r_1)}-e^{A(s-r_1)})\|}{(r_2-r_1)^\beta}\\
&= \sup_{0\leq r_1<r_2<s}\frac{\|(e^{A(t-s)}-{\rm id})(e^{A(s-r_1)}-e^{A(s-r_2)})\|}{(r_2-r_1)^\beta}\\
&\leq M^2 \|A\|^2 (t-s) s^{1-\beta}.
\end{split}
\end{align}

\smallskip

We finish the section by presenting a Gronwall-like lemma:
\begin{lemma}\label{l1b}
For $0<\eps<\lambda$ let $\hat\eps>0$ be a number such that
\begin{equation}\label{eq15}
  e^{-\lambda}+\hat\eps\le (1+\hat\eps)e^{-(\lambda-\eps)}.
\end{equation}
Let $(v_ n)_{n\geq 0}$ be a sequence of positive numbers such that
\begin{equation}\label{eq5b}
  v_n\le k \zeta_0e^{-\lambda n}+\hat\eps \sum_{j=0}^{n-1} v_je^{-\lambda(n-j-1)}
\end{equation}

where $\zeta_0,\,k$ are positive numbers\footnotemark. \footnotetext{For $j>n-1$ the previous sum is zero.}Then for $n\ge 0$ we have
\begin{equation*}
  v_n\le (1+\hat\eps)^{n}e^{-n(\lambda-\eps)}k\zeta_0.
\end{equation*}
\end{lemma}
\begin{proof}
Let us denote the right hand side of \eqref{eq5b} by $Z_n$. We are going to show that
\begin{equation}\label{Z}
  Z_n\le (1+\hat\eps)^{n}e^{-n(\lambda-\eps)}k\zeta_0.
\end{equation}
Since $Z_0=k\zeta_0$ the result holds for $n=0$.
For $n=1$, thanks to the assumption (\ref{eq15}) we obtain
\begin{equation*}
v_1\le Z_1= k\zeta_0 e^{-\lambda}+\hat\eps v_0\le  k\zeta_0 e^{-\lambda}+\hat\eps  k\zeta_0  \le (1+\hat\eps)e^{-(\lambda-\eps)} k\zeta_0.
\end{equation*}
Suppose (\ref{Z}) holds for $n\ge 1$. Taking into account that $Z_{n+1}= e^{-\lambda}Z_n+\hat \eps v_n\leq (e^{-\lambda}+\hat \eps)Z_n$, we have
\begin{equation*}
  v_{n+1}\le Z_{n+1}\leq (e^{-\lambda}+\hat \eps)Z_n \le (1+\hat \eps)e^{-(\lambda-\eps)}Z_n\le (1+\hat\eps)^{n+1}e^{-(n+1)(\lambda-\eps)}k\zeta_0.
\end{equation*}
\end{proof}

\section{Integrals for a H{\"o}lder--continuous integrator with H\"older exponent greater than 1/2}\label{s1}

In this section, we present the Young-integral having a H\"older--continuous function with H\"older exponent greater than 1/2 as integrator. To be more precise, let $T>0$ and consider a mapping
\begin{equation*}
  g:[0,T]\to L(\RR^m,\RR^d)
\end{equation*}
such that $g\in C^\beta([0,T]; L(\RR^m,\RR^d))$.
Assuming that $\beta+\beta^\prime>1$ we can define the Young-integral with integrand $g$ and integrator $\omega \in C^{\beta^\prime}([0,T];\RR^m)$
\begin{equation*}
  \int_s^t gd\omega
\end{equation*}
for $0\leq s<t\leq T$, see \cite{You36}. Furthermore, one can represent this integral in terms of fractional derivatives: for $\alpha\in (0,1)$ we define
\begin{align*}\label{fractder}
 \begin{split}
    D_{{s}+}^\alpha g[r]=&\frac{1}{\Gamma(1-\alpha)}\bigg(\frac{g(r)}{(r-s)^\alpha}+\alpha\int_{s}^r\frac{g(r)-g(q)}{(r-q)^{1+\alpha}}dq\bigg),\\
    D_{{t}-}^{1-\alpha} \omega_{{t}-}[r]=&\frac{(-1)^{1-\alpha}}{\Gamma(\alpha)}
    \bigg(\frac{\omega(r)-\omega(t)}{(t-r)^{1-\alpha}}+(1-\alpha)\int_r^{t}\frac{\omega(r)-\omega(q)}{(q-r)^{2-\alpha}}dq\bigg),
\end{split}
\end{align*}
where $\omega_{{t}-}(\cdot)=\omega(\cdot)-\omega(t)$. Under the condition $\beta+\beta^\prime>1$, there exists an $\alpha$ such that $\alpha<\beta,\,\alpha+\beta^\prime>1$, and these inequalities ensure that the above operators exist. Then the Young-integral can be expressed as
\begin{equation}\label{eq12}
  \int_s^tgd\omega=(-1)^\alpha\int_s^tD_{{s}+}^\alpha g[r]D_{{t}-}^{1-\alpha} \omega_{{t}-}[r]dr,
  \end{equation}
see for instance \cite{Zah98}. Taking into account the definition of the fractional derivatives, it is easy to derive the following estimate for the above integral:
\begin{equation}\label{eq11}
  \bigg\|\int_s^tgd\omega\bigg\|\le C_{\alpha,\beta,\beta^\prime} \ltn \omega \rtn_{\beta^\prime,s,t}\bigg(\|g\|_{\infty,s,t}+(t-s)^{\beta}\ltn g \rtn_{\beta,s,t}\bigg)(t-s)^{\beta^\prime},
\end{equation}
which in particular implies that
\begin{equation*}
  [0,T]\ni t\mapsto  \int_0^tgd\omega\in C^{\beta^\prime}([0,T];\RR^d),
\end{equation*}
with
\begin{equation*}
  \bigg\|\int_0^tgd\omega\bigg\|_{\beta^\prime,0,T}\le C_{\alpha,\beta,\beta^\prime,T}\|g\|_{\beta,0,T} \ltn \omega \rtn_{\beta^\prime,0,T}.
\end{equation*}
For $T\le 1$ we shall denote $C_{\alpha,\beta, \beta^\prime,T}$ by $C_{\alpha,\beta,\beta^\prime}$ in the following.

We also know that the integral is additive: let $s\le \tau\le t$, then
\begin{equation*}
  \int_s^\tau gd\omega+\int_\tau^t gd\omega=\int_s^t gd\omega,
\end{equation*}
see  \cite{Zah98}, and for any linear operator $L:\RR^d\to\RR^m$
\begin{equation*}
  L\int_s^tgd\omega=\int_s^t(Lg)d\omega=\int_s^tLgd\omega
\end{equation*}
because $LD_{s+}^\alpha g= D_{s+}^\alpha Lg$.

Finally, for the Wiener shift flow  $\theta=(\theta_t)_{t\in\RR}$ given by (\ref{shift}) the following shift property of the integral holds:

\begin{lemma}\label{l4}
Let $\RR\ni t\mapsto \omega(t) \in \RR$ have a finite $\beta^\prime$--H{\"o}lder--norm for any closed finite subinterval of $\RR$ and similar for $ \RR \ni t\mapsto g(t) \in \RR$ with respect to the $\beta$--H{\"o}lder--norm, with $\beta+\beta^\prime>1$. Then
\begin{equation*}
  \int_{s+\tau}^{t+\tau}gd\omega=\int_s^tg(\cdot+\tau)d\theta_\tau\omega.
\end{equation*}
\end{lemma}
\begin{proof}
For $1-\beta^\prime <\alpha<\beta$ we have
\begin{align*}
  D^{1-\alpha}_{t-}(\theta_\tau\omega)_{t-}[r]  =&\frac{(-1)^{1-\alpha}}{\Gamma(\alpha)}\bigg(\frac{\theta_\tau\omega(r)-\theta_\tau\omega(t)}{(t-r)^{1-\alpha}}+(1-\alpha)\int_{r}^t\frac{\theta_\tau\omega(r)-\theta_\tau\omega(q)}{(q-r)^{2-\alpha}}dq\bigg)\\
  =&\frac{(-1)^{1-\alpha}}{\Gamma(\alpha)}\bigg(\frac{ \omega(r+\tau)-\omega(t+\tau)}{(t-r)^{1-\alpha}}+(1-\alpha)\int_{r}^t\frac{\omega(r+\tau)-\omega(q+\tau)}{(q-r)^{2-\alpha}}dq\bigg)\\
  =&\frac{(-1)^{1-\alpha}}{\Gamma(\alpha)}\bigg(\frac{ \omega(r+\tau)-\omega(t+\tau)}{(t+\tau-(r+\tau))^{1-\alpha}}+(1-\alpha)\int_{r+\tau}^{t+\tau}\frac{\omega(r+\tau)-\omega(q)}{(q-(r+\tau))^{2-\alpha}}dq\bigg)\\
  =&D^{1-\alpha}_{(t+\tau)-}\omega_{(t+\tau)-}[r+\tau]
\end{align*}
and similar for $D_{s+}^\alpha g(\cdot+\tau)[r]=D_{(s+\tau)+}^\alpha g[r+\tau]$. It suffices to execute the variable transform $r \mapsto r+\tau$ in the integral.
\end{proof}

\medskip

The Young--integral introduced above can be applied to define pathwise stochastic integrals for the fractional Brownian motion $B^H$ with Hurst--parameter $H\in (1/2,1)$. In particular $B^H$ can be replaced by the canonical fractional Brownian motion in \eqref{eq5} which is H{\"o}lder--continuous with $\PP_H$ probability one.

\section{Differential equations driven by H{\"o}lder--continuous paths with H\"older-exponent larger than 1/2}\label{s2}

Let $F:\RR^d\to\RR^d$ and $G:\RR^d\to L(\RR^m,\RR^d)$, and let $\omega$ be a noisy input, considered as a function from $\RR^+$ to $\RR^m$. Then, for $T>0$ consider the equation
\begin{equation}\label{eq3}
  du(t)=F(u(t))dt+G(u(t))d\omega(t)
\end{equation}
with initial condition $u(0)=u_0\in\RR^d$. This equation is interpreted as
\begin{equation}\label{eq4}
  u(t)=u_0+\int_0^tF(u(r))dr+\int_0^tG(u(r))d\omega(r),\quad t\in [0,T],
\end{equation}
where the first integral is defined as a standard Riemann--integral while the second one is defined as the Young--integral introduced in Section \ref{s1}.

We will use the following assumptions on $F$ and $G$:
\begin{itemize}
 \item[(A1)] $F:\RR^d\to\RR^d$ is  continuously differentiable with bounded derivative,
 \item[(A2)] $G:\RR^d\to L(\RR^m,\RR^d)$ is twice continuously differentiable with bounded derivatives.
\end{itemize}

Regarding the existence of solutions, the next result follows by \cite{NuaRas02}, although with a slight modification of the phase spaces that appear in that reference; see also \cite{BH}, but notice that in this last paper a delay equation is considered, and therefore in our setting we should take the delay equal to zero.

\begin{theorem}\label{t1}
Suppose (A1) and (A2). If $\omega\in C^{\beta^\prime}([0,T];\RR^m)$ with $\beta^\prime > \beta>1/2$, then \eqref{eq3} (or equivalently \eqref{eq4}) has a unique solution $u\in C^\beta([0,T];\RR^d)$  for any $T>0$.
\end{theorem}
Note that the derivatives of $G$ are defined as follows
\begin{align*}
   &DG:\RR^d\to L(\RR^d,L(\RR^m,\RR^d))\equiv L(\RR^d\times \RR^m,\RR^d),\\
   &D^2G:\RR^d\to L(\RR^d,L(\RR^d,L(\RR^m,\RR^d))\equiv L(\RR^d\times\RR^d\times \RR^m,\RR^d).
\end{align*}

We define a matrix $A \in \RR^{d\times d}$ and the function $\hat F:\RR^d\to\RR^d$ by
$$A=DF(0), \qquad \hat F(x)=F(x)-Ax, \quad \text{ for } x\in  \RR^d.$$
Then $x\mapsto \hat F(x)$ and $x\mapsto D\hat F(x)=DF(x)-DF(0)$ are continuous, and $D\hat F(0)=0$.

We also will need further the assumptions:
\begin{itemize}
 \item[(A3)] $F(0)=0$, $G(0)=0$,
 \item[(A4)] $DG(0)=0$.
\end{itemize}

\smallskip

We consider then the following equation
\begin{align}\label{eq2}
\begin{split}
  u(t)=&e^{At}u_0+\int_0^te^{A(t-r)}\hat F(u(r))dr+\int_0^te^{A(t-r)}G(u(r))d\omega(r),
\end{split}
\end{align}
where the last integral has to be understood as in Section \ref{s1}.
\begin{lemma}\label{l5}
Let $T>0$ and assume  (A1) and (A2). If $\omega\in C^{\beta^\prime}([0,T]; \RR^m)$ with $\beta^\prime > \beta>1/2$, then equation (\ref{eq2}) has a unique solution $u\in C^\beta([0,T];\RR^d)$ that also agrees with (\ref{eq4}). Furthermore, if we also assume (A3), equation (\ref{eq2}) possesses the trivial solution $u=0$.

\end{lemma}
\begin{proof}
In view of the regularity of $\hat F$ and $G$, the existence and uniqueness of a solution to (\ref{eq2}) follows by \cite{GMS08}. Now we want to prove that such a solution coincides with the solution of (\ref{eq4}). To achieve such a result, notice that when $\omega$ is a sufficiently smooth path, then (\ref{eq4}) and (\ref{eq2}) are the same solutions
using classical calculus.

Now it suffices to follow an approximation argument. To be more precise, consider \eqref{eq2} for a sequence of driving paths $(\omega^n)_{n\in\NN}$ which are given by the piecewise linear interpolation of $w$ with stepsize $T2^{-n}$. Then the sequence $(u^n)_{n\in\NN}$ related to these
piecewise linear paths converges to the solution of (\ref{eq4}) and (\ref{eq2}) as well, being both of them driven by $\omega$, see Chapter 10 in \cite{FV10}.
Therefore both solutions are the same.
\end{proof}

Note that a sufficient condition for the convergence of $(\omega^n)_{n\in\NN}$ to $\omega$ is that $\omega\in C^{\beta^{\prime\prime}}([0,T];\RR^m)$ for $\beta^{\prime}<\beta^{\prime\prime}\le 1$.\\

In the following we can restrict the mappings $F$ and $G$ to be defined on some neighborhood of zero. For $\rho>0$ assume:
\begin{itemize}
 \item[(A1)'] $F: \bar B(0,\rho)\subset \RR^d\to \RR^d$ is  continuously differentiable with bounded derivative,
 \item[(A2)'] $G:\bar B(0,\rho)\subset\RR^d\to L(\RR^m,\RR^d)$ is twice continuously differentiable with bounded derivatives.
\end{itemize}

In order to look at the {\em local} asymptotic behavior of (\ref{eq2}), we define $\chi$ to be the cut--off function
\begin{equation}\label{eq20}
  \chi:\RR^d\to \bar B(0,1)\subset \RR^d \quad \textrm{where} \quad \chi(u)=\left\{\begin{array}{lcl}
  u& \textrm{if} & \|u\|\le \frac12\\
  0& \textrm{if} & \|u\|\ge 1
  \end{array}
   \right. .
\end{equation}
In particular the norm of $\chi(u)$ is bounded by 1. Let us assume that $\chi$ is twice continuously differentiable with bounded derivatives $D\chi$ and $D^2 \chi$. Let us denote by $L_{D\chi},\,L_{D^2\chi}$ the bounds for those derivatives. Now for $u\in \RR^d$ and some $0<\hat R\le \rho$ we set
$$\chi_{\hat R}(u)=\hat R \cdot \chi(u/\hat R)\in \bar B(0,\hat R).$$
Then it is not difficult to see that its first derivative $D\chi_{\hat R}$ is bounded by $L_{D\chi}$, while the second derivative $D^2\chi_{\hat R}$ is bounded by $ {L_{D^2\chi}} / \hat R$.

Define the functions $\hat F_{\hat R}:=\hat F\circ \chi_{\hat R}:\RR^d \to \RR^d$ and $G_{\hat R}:=G\circ \chi_{\hat R}:\RR^d \to L(\RR^m,\RR^d)$. Replacing $\hat F$ by $\hat F_{\hat R}$ and $G$ by $G_{\hat R}$ in (\ref{eq2}), we obtain a unique solution of the corresponding equation (\ref{eq2}). This statement follows by Lemma \ref{l5}, since $\hat F_{\hat R}$ is continuously differentiable with bounded derivative and $G_{\hat R}$ is twice differentiable with bounded first and second derivatives.

\smallskip

\begin{lemma}\label{l6} Assume (A1)', (A2)', (A3) and (A4). Then for every $R>0$ there exists a positive $\hat R\le \rho$ such that for $u, \, z\in \RR^d$
\begin{equation}\label{eq7}
  \|\hat F_{\hat R}(u)\|\le R L_{D\chi}\|u\|,
  \end{equation}
\begin{align}\label{eq23}
 & \|G_{\hat R}(u)\|\le R L_{D\chi} \|u\|,
\end{align}
\begin{align}\label{eq23b}
& \|G_{\hat R}(u)-G_{\hat R}(z)\|\le R L_{D\chi}\|u-z\|.
\end{align}

\end{lemma}
\begin{proof}
Since $D\hat F$ and $DG$ are continuous, being $DF(0)=0$ and $DG(0)=0$, for any $R>0$ we can choose an $\hat R\le \rho$ such that
\begin{equation*}
  \sup_{\|v\|\le \hat R}\|D\hat F(v)\|\le R \qquad \textrm{and}  \qquad  \sup_{\|v\|\le \hat R}\|D G(v)\|\le R.
\end{equation*}
Then for $u\in \RR^d$ we have
\begin{align*}
  \|\hat F_{\hat R}(u)\|&\le \sup_{z\in \RR^d}\|D(\hat F(\chi_{\hat R}(z)))\| \|u\| \le  \sup_{\|v\|\le \hat R}\|D\hat F(v)\|\sup_{z\in \RR^d}\|D\chi_{\hat R}(z)\|\|u\|\\
  &\le  \sup_{\|v\|\le \hat R}\|D\hat F(v)\|L_{D\chi}\|u\|\le R L_{D\chi}\|u\|,
  \end{align*}
and we obtain (\ref{eq7}). Due to the fact that $G(0)=0$ we can follow the same steps to prove (\ref{eq23}).

Finally, due to the regularity of $G$, we have
\begin{align*}
\|G_{\hat R}(u)-G_{\hat R}(z)\|& \le \sup_{\|v\|\le \hat R}\|DG(v)\|\| \chi_{\hat R} (u)-\chi_{\hat R} (z)\|\\
& \le L_{D\chi} \sup_{\|v\|\leq \hat R}\|DG(v)\| \|u-z\| \leq R L_{D\chi} \|u-z\|.
\end{align*}
\end{proof}

\section{Local exponential asymptotic stability}\label{s3}

In order to study the large time behavior of (\ref{eq2}), we consider iteratively a family of differential equations for $n\in\ZZ^+$ defined on the fixed interval $[0,1]$. In particular we analyze the following sequence of equations with driving function $\theta_n\omega$ and a coefficient $\hat R$ depending on $\theta_n\omega$:
\begin{align*}
  u^n(t)=e^{At}u^n(0)& +\int_0^t e^{A(t-r)}\hat F_{\hat R(\theta_n\omega)}(u^n(r))dr\\
  &+\int_0^t e^{A(t-r)}G_{\hat R(\theta_n\omega)}(u^n(r))d\theta_n\omega(r),\qquad t \in [0,1].
\end{align*}
We set $u^0(0)=u_0$ and $u^n(0)=u^{n-1}(1)$ for $n\in \NN$. Under (A1)' and (A2)' the functions $\hat F_{\hat R}$, $G_{\hat R}$ satisfy the conditions of Lemma \ref{l5}, so for any $n\in \NN$ each one of the above problems has a unique solution $u^n \in  C^\beta([0,1];\RR^d)$.

\smallskip

In what follows we want to estimate the H\"older--norm of each solution $u^n$ on $[0,1]$. The next assumption we need is:
\begin{itemize}
 \item[(A5)] We have  ${\rm Re}\,\sigma(A)<-\lambda<0$ for $A=DF(0)$.
\end{itemize}

Regarding the standard Riemann--integral, by Lemma \ref{l1} and \eqref{eq7} we have
\begin{align*}
\bigg\|\int_0^\cdot e^{A(\cdot-r)}\hat F_{\hat R(\theta_n\omega)}(u^n(r))dr\bigg\|_{\infty,0,1}\le MR(\theta_n \omega) L_{D\chi} \|u^n\|_{\infty,0,1},
  \end{align*}
with the relationship between $R(\omega)$ and $\hat R(\omega)$ given in Lemma \ref{l6}. Furthermore, for the H{\"o}lder--seminorm, thanks to (\ref{semi}),
\begin{align*}
  &\ltn\int_0^\cdot e^{A(\cdot-r)} \hat F_{\hat R(\theta_n\omega)}(u^n(r))dr\rtn_{\beta,0,1}\\
 & =\sup_{0\leq s<t\leq 1} \frac{\bigg \|\int_s^te^{A(t-r)}\hat F_{\hat R(\theta_n\omega)}(u^n(r))dr+\int_0^s (e^{A(t-r)}-e^{A(s-r)})F_{\hat R(\theta_n\omega)}(u^n(r))dr\bigg\|}{(t-s)^{\beta}}\\
  &\le \sup_{0\leq s<t\leq 1}  \bigg((t-s)^{1-\beta}\sup_{r\in[s,t]}(\|e^{A(t-r)}\|\|\hat F_{\hat R(\theta_n\omega)}(u^n(r))\|)\bigg)\\
  &\,\, +\sup_{0\leq s<t\leq 1} \bigg(\frac{s}{(t-s)^\beta} \sup_{r\in[0,s]}(\|e^{A(t-r)}-e^{A(s-r)}\|\|\hat F_{\hat R(\theta_n\omega)}(u^n(r))\|)\bigg)\\
&  \le M R(\theta_n\omega) L_{D\chi}\|u^n\|_{\infty,0,1}
 +M  \|A\|R(\theta_n\omega) L_{D\chi} \|u^n\|_{\infty,0,1}  \\
& \le  M(1+\|A\|)R(\theta_n\omega) L_{D\chi} \|u^n\|_{\infty,0,1},
\end{align*}
and therefore
\begin{align}\label{eq26a}
\begin{split}
 \bigg \|\int_0^\cdot   & e^{A(\cdot-r)} \hat F_{\hat R(\theta_n\omega)}(u^n(r))dr \bigg \|_{\beta,0,1} \leq M(2+\|A\|)R(\theta_n\omega) L_{D\chi}\|u^n\|_{\beta,0,1}.
\end{split}
\end{align}

Now we estimate the $\beta$--H{\"o}lder--norm of the integral containing $G_{\hat R}$. Choose a $0<\alpha<1/2,\,\alpha+\beta^\prime>1$ and assume that $0\le s<t\le 1$. Then from (\ref{eq11})
\begin{align}\label{eq22}
\begin{split}
  &\bigg\|\int_0^t e^{A(t-r)}G_{\hat R(\theta_n\omega)}(u^n(r))d\theta_n\omega(r)-\int_0^s e^{A(s-r)}G_{\hat R(\theta_n\omega)}(u^n(r))d\theta_n\omega(r) \bigg\|\\
  & \le\bigg\|\int_s^t e^{A(t-\cdot)}G_{\hat R(\theta_n\omega)}(u^n(r))d\theta_n\omega(r) \bigg\|\\
  & \, \, +\bigg\| \int_0^s (e^{A(t-r)}-e^{A(s-r)}) G_{\hat R(\theta_n\omega)}(u^n(r))d\theta_n\omega(r) \bigg\|\\
  & \le C_{\alpha,\beta,\beta^\prime} \ltn \theta_n\omega \rtn_{\beta^\prime} \|e^{A(t-\cdot)}G_{\hat R(\theta_n\omega)}(u^n(\cdot))\|_{\beta,0,t}(t-s)^{\beta^\prime}\\
  & \, \, + C_{\alpha,\beta,\beta^\prime}  \ltn \theta_n\omega \rtn_{\beta^\prime} \|(e^{A(t-\cdot)}-e^{A(s-\cdot)})G_{\hat R(\theta_n\omega)}(u^n(\cdot))\|_{\beta,0,s}  s^{\beta^\prime}.
\end{split}
\end{align}
Since for two any $\beta$--H{\"o}lder--continuous functions $f,\,g$ we have
\begin{equation*}
  \|fg\|_{\beta,0,t}\le \|f\|_{\infty,0,t}\|g\|_{\beta,0,t}+\|g\|_{\infty,0,t}\ltn f\rtn_{\beta,0,t},
\end{equation*}
it follows
\begin{align*}
  \|e^{A(t-\cdot)}G_{\hat R(\theta_n\omega)}(u^n(\cdot))\|_{\beta,0,t}\le& \|e^{A(t-\cdot)}\|_{\infty,0,t}\|G_{\hat R(\theta_n\omega)}(u^n(\cdot))\|_{\beta,0,t}\\
  &+\|G_{\hat R(\theta_n\omega)}(u^n(\cdot))\|_{\infty,0,t}\ltn e^{A(t-\cdot)}\rtn_{\beta,0,t}.
\end{align*}
Thanks to (\ref{eq23}) and (\ref{eq23b}) we obtain
\begin{align*}
\|G_{\hat R(\theta_n\omega)}(u^n(\cdot))\|_{\beta,0,t}&=\sup_{s\in[0,t]}\|G_{\hat R(\theta_n\omega)}(u^n(s))\|\\
&+\sup_{0\leq r_1<r_2\leq t} \frac{\|G_{\hat R(\theta_n\omega)}(u^n(r_2))-G_{\hat R(\theta_n\omega)}(u^n(r_1))\|}{(r_2-r_1)^\beta}\\
&\leq R(\theta_n\omega) L_{D\chi}\bigg(\|u^n\|_{\infty,0,t}+\sup_{0\leq r_1<r_2\leq t} \frac{\|u^n(r_2)-u^n(r_1)\|}{(r_2-r_1)^\beta}\bigg)\\
&=  R(\theta_n\omega) L_{D\chi} \|u^n\|_{\beta,0,t},
\end{align*}
hence, taking into account Lemma \ref{l1} and (\ref{sem}), it follows that
\begin{align*}
  \|e^{A(t-\cdot)}G_{\hat R(\theta_n\omega)}(u^n(\cdot))\|_{\beta,0,t}&\le MR(\theta_n\omega) L_{D\chi} \|u^n\|_{\beta,0,1}+M\|A\|R(\theta_n\omega) L_{D\chi} \|u^n\|_{\infty,0,1}\\
  &\leq M R(\theta_n\omega) L_{D\chi}(1+\|A\|)\|u^n\|_{\beta,0,1}.
\end{align*}
In a similar way, using (\ref{semi}) and (\ref{sem1}) in particular we obtain
\begin{align*}
  \|(e^{A(t-\cdot)}-e^{A(s-\cdot)})G_{\hat R(\theta_n\omega)}(u^n(\cdot))\|_{\beta,0,s}&\le  M\|A\| R(\theta_n\omega) L_{D\chi}\\ &\qquad \qquad \times  (1+M\|A\|)\|u^n\|_{\beta,0,1} (t-s),
\end{align*}
and therefore
\begin{align*}
\ltn \int_0^\cdot e^{A(\cdot-r)}G_{\hat R(\theta_n\omega)}(u^n(r))d\theta_n\omega(r) \rtn_{\beta,0,1}&\leq C_{\alpha,\beta,\beta^\prime}  \ltn \theta_n\omega \rtn_{\beta^\prime} M R(\theta_n\omega) L_{D\chi}\\
&\qquad \qquad \times(1+2\|A\|+M\|A\|^2)  \|u^n\|_{\beta,0,1}.
\end{align*}
Using the same kind of calculations we get
\begin{align*}
  \bigg\|\int_0^\cdot e^{A(t-\cdot)} G_{\hat R(\theta_n\omega)}(u(r))d\theta_n\omega(r) \bigg\|_{\infty,0,1}&\leq C_{\alpha,\beta,\beta^\prime}  \ltn \theta_n\omega \rtn_{\beta^\prime}  M R(\theta_n\omega) L_{D\chi}
 \\  &\qquad \qquad \times
 (1+\|A\|)\|u^n\|_{\beta,0,1}.
\end{align*}
Collecting these estimates we have
\begin{align*}
 \bigg \|\int_0^\cdot & e^{A(\cdot-r)}G_{\hat R(\theta_n\omega)}(u(r))d\theta_n\omega(r) \bigg\|_{\beta,0,1}\leq K  \ltn \theta_n\omega \rtn_{\beta^\prime,0,1}R(\theta_n\omega)\|u^n\|_{\beta,0,1},
 \end{align*}
 where
 \begin{equation}\label{k}
 K=\max\{1,C_{\alpha,\beta,\beta^\prime} \}M^2 L_{D\chi}(2+3\|A\|+\|A\|^2)
 \end{equation}
 using that $M \geq 1$.
Note that the constant $K$ can be also used to estimate the constant $M L_{D\chi}(2+\|A\|)$ in \eqref{eq26a}, i.e. we have
\begin{align}\label{eq26}
\begin{split}
 \bigg \|\int_0^\cdot   & e^{A(\cdot-r)} \hat F_{\hat R(\theta_n\omega)}(u^n(r))dr \bigg \|_{\beta,0,1} \leq
 K  R(\theta_n\omega)\|u^n\|_{\beta,0,1}.
\end{split}
\end{align}

Define now the function
\begin{equation}\label{eq8}
  u(t)=u^n(t-n)\quad \text{if}\,\, t\in [n,n+1].
\end{equation}
{On account of Lemma  \ref{l4}, for $t\in [n,n+1]$ we have
\begin{align*}
  u(t) & = e^{A(t-n)}u(n) +\int_n^te^{A(t-r)}\hat F_{\hat R(\theta_n\omega)}(u(r))dr+\int_n^te^{A(t-r)} G_{\hat R(\theta_n\omega)}(u(r))d\omega(r)\\
 & =  e^{At}u_0+\sum_{j=0}^{n-1}e^{A(t-j-1)} \int_{j}^{j+1}e^{A(j+1-r)}\hat F_{\hat R(\theta_j\omega)}(u(r))dr \\ & \quad   + \sum_{j=0}^{n-1}e^{A(t-j-1)}  \int_{j}^{j+1}e^{A(j+1-r)} G_{\hat R(\theta_j\omega)}(u(r))d\omega(r) \\
  & \quad  +\int_{n}^{t}e^{A(t-r)}\hat F_{\hat R(\theta_j\omega)}(u(r))dr+\int_n^{t}e^{A(t-r)} G_{\hat R(\theta_j\omega)}(u(r))d\omega(r)\\
  &=e^{At}u_0+\sum_{j=0}^{n-1}e^{A(t-j-1)} \int_0^1e^{A(1-r)}\hat F_{\hat R(\theta_j\omega)}(u^j(r))dr \\ & \quad + \sum_{j=0}^{n-1}e^{A(t-j-1)}\int_0^1e^{A(1-r)} G_{\hat R(\theta_j\omega)}(u^j(r))d\theta_j\omega(r)  \\
  & \quad +\int_{0}^{t-n}e^{A(t-n-r)}\hat F_{\hat R(\theta_n\omega)}(u^n(r))dr+\int_0^{t-n}e^{A(t-n-r)} G_{\hat R(\theta_n\omega)}(u^n(r))d\theta_n\omega(r).
\end{align*}}
Note that the $\beta$--H\"older norm of the last two terms of the above expression can be calculated following the previous estimates. However, the terms under the sum can be estimated even in a simpler way, since
\begin{align*}
 \bigg\| &e^{A(\cdot-j-1)} \int_{0}^{1}e^{A(1-r)} G_{\hat R(\theta_j\omega)}(u^j(r))d\theta_j\omega(r)\bigg \|_{\beta,n,n+1}\\
   &\leq \|e^{A(\cdot-j-1)} \|_{\beta,n,n+1} \bigg\| \int_{0}^{\cdot}e^{A(\cdot-r)} G_{\hat R(\theta_j\omega)}(u^j(r))d\theta_j\omega(r)\bigg \|_{\infty,0,1},
\end{align*}
and from Lemma \ref{l1}, it is simple to obtain that
$$ \|e^{A(\cdot-j-1)} \|_{\beta,n,n+1} \leq M(1+\|A\|)e^{-\lambda(n-j-1)},$$
giving us
\begin{align*}
 \bigg\| &e^{A(\cdot-j-1)} \int_{0}^{1}e^{A(1-r)} G_{\hat R(\theta_j\omega)}(u^j(r))d\theta_j\omega(r)\bigg \|_{\beta,n,n+1}\\
 &\leq M^2(1+\|A\|)^2 e^{-\lambda(n-j-1)}C_{\alpha,\beta,\beta^\prime}  \ltn \theta_j\omega \rtn_{\beta^\prime}   R(\theta_j\omega) L_{D\chi}\|u^j\|_{\beta,0,1}\\
 &\leq K e^{-\lambda(n-j-1)} \ltn \theta_j\omega \rtn_{\beta^\prime}   R(\theta_j\omega)\|u^j\|_{\beta,0,1},
 \end{align*}
where $K$ has been introduced in (\ref{k}). Following similar steps,
\begin{align*}
 \bigg\| &e^{A(\cdot-j-1)} \int_{0}^{1}e^{A(1-r)}\hat F_{\hat R(\theta_j\omega)}(u^j(r))dr \bigg \|_{\beta,n,n+1}\\
 &\leq M^2(1+\|A\|) L_{D\chi}e^{-\lambda(n-j-1)} R(\theta_j\omega)\|u^j\|_{\beta,0,1}\\
  &\leq K e^{-\lambda(n-j-1)} R(\theta_j\omega)\|u^j\|_{\beta,0,1}.
\end{align*}
Hence
\begin{align*}
\begin{split}
  \|u^n\|_{\beta,0,1}&\le  \|u_0\|  \|e^{A\cdot}\|_{\beta,n,n+1}+K\sum_{j=0}^{n-1}
   R(\theta_j\omega)(1+\ltn \theta_j\omega \rtn_{\beta^\prime,0,1})\|u^j\|_{\beta,0,1}e^{-\lambda(n-j-1)}\\
   &+K R(\theta_n\omega) (1+\ltn \theta_n\omega \rtn_{\beta^\prime,0,1}) \|u^n\|_{\beta,0,1}. \\
\end{split}
\end{align*}

Let $\eps,\,\hat\eps$ given by Lemma \ref{l1b}, assuming in addition that $\hat\eps<1$ and
\begin{equation}\label{cond}
\log(1+\hat \eps)\leq \lambda-\eps.
\end{equation}

Consider
\begin{equation*}
  R(\omega)=\frac{\hat \eps}{2K(1+\ltn \omega \rtn_{\beta^\prime,0,1})}
\end{equation*}
and define $\hat R(\omega)$ such that
\begin{align}\label{Rhat}
  &\hat R(\omega)=\max\bigg\{\hat r:\sup_{\|v\| \leq \hat r}(\|DG(v)\|+\|D\hat F(v)\|)\le  R(\omega)\bigg\}\wedge \rho.
\end{align}
With this choice, the coefficient in front of $\|u^n\|_{\beta,0,1}$ on the right hand side of the above expression is less than or equal $1/2$, since $\hat \eps<1$. As a consequence,
\begin{align*}
  \frac{1}{2}\|u^n\|_{\beta,0,1}\le &  \|u_0\|  \|e^{A\cdot}\|_{\beta,n,n+1} +\frac{\hat \eps}{2} \sum_{j=0}^{n-1} e^{-\lambda (n-j-1)} \|u^j\|_{\beta,0,1}
\end{align*}
and hence
\begin{align*}
\|u^n\|_{\beta,0,1}  \le& 2 M(1+\|A\|) \|u_0\|e^{-\lambda n}+\hat\eps \sum_{j=0}^{n-1}e^{-\lambda(n-j-1)}  \|u^j\|_{\beta,0,1}.
  \end{align*}
Taking $v_j=\|u^j\|_{\beta,0,1}$, $\zeta_0=\|u_0\|$ and $k=2 M(1+\|A\|)$, Lemma \ref{l1} ensures that
\begin{equation*}
\|u^n\|_{\beta,0,1}\le (1+\hat\eps)^n e^{-n(\lambda-\eps)} k \|u_0\|\le k \|u_0\| e^{-n((\lambda-\eps)-\log(1+\hat\eps))}.
\end{equation*}

Note that $\hat R$ defined above is measurable, see Lemma \ref{l7}. Furthermore, if $
\ltn\omega \rtn_{\beta^\prime,0,1}$
satisfies
\begin{align}
\label{ass-w}
\lim_{t \rightarrow \pm \infty} \frac{ \log^+ \ltn \theta_t \omega \rtn_{\beta^\prime,0,1}}{t} =0,
\end{align}
then  according to Lemma \ref{l2} it follows that for a sufficiently small $\eps>0$ there exists $C_\eps(\omega)$ and $\kappa\in (0,\infty]$ such that
\begin{equation*}
  \hat R(\theta_t\omega)\ge \frac{\kappa}{2}R(\theta_t\omega)\ge \frac{\kappa}{2} C_\eps(\omega) e^{-\eps |t|}
\end{equation*}
for sufficiently large $|t|$. If $\omega$ is a sample path of a stochastic process, this means that $R$ and $\hat R$ are tempered from below. Condition \eqref{ass-w} holds in particular if
 $\omega$ is a sample path of the canonical fractional Brownian motion with Hurst-parameter $H>1/2$ defined on the probability space introduced in Section  \ref{s0}.

 Now, due to (\ref{cond}) and Lemma \ref{l8}, we can find a zero neighborhood depending on $\omega$  such that for $u_0$ contained in this neighborhood we have
\begin{equation*}
  \|u^n\|_{\beta,0,1}\le \frac{\hat R(\theta_n\omega)}{2}\quad \text{for all } n\in\ZZ^+.
\end{equation*}
Then we have
\begin{equation*}
  \hat F_{\hat R(\theta_n\omega)}(u^n(r))=\hat F(u^n(r)),\quad G_{\hat R(\theta_n\omega)}(u^n(r))=G(u^n(r))
\end{equation*}
for $r\in [0,1]$ and $n\in \ZZ^+$.  Then we see that $u$ defined by \eqref{eq8} solves \eqref{eq2} on any interval $[0,T]$. Furthermore, we have obtained the following stability result for the solution $u$:
\smallskip

\begin{theorem}\label{t3} 
Suppose that $\omega\in C_0^{\beta^\prime}(\RR;\RR^d)$ and that conditions \eqref{ass-w}, $(A1)'-(A2)'$ and $(A3)-(A5)$ hold. 
Then for every $\hat \eps$, $\eps$ and $\lambda$ satisfying (\ref{eq15}) and (\ref{cond}), there exists a  neighborhood  of zero (depending on $\omega$) such that if $u_0$
is contained in this neighborhood the solution of \eqref{eq2} with initial condition $u_0$ is exponentially stable with an exponential rate less than or equal to $(\lambda-\eps)-\log(1+\hat \eps)$.
\end{theorem}

\begin{remark}
 The assumptions on $\omega$ of the previous Theorem are in particular satisfied by $\PP_H$-almost all sample paths
 of the canonical $H$-fractional Brownian motion.
\end{remark}

\begin{remark}
The solution of (\ref{eq2}) is locally exponential stable with any rate less than $\lambda$. Indeed, for any arbitrary $\mu <\lambda$ we can find $\eps$ and $\hat \eps$ satisfying (\ref{eq15}) and (\ref{cond}) such that
$$\lambda>\lambda-\eps-\log(1+\hat \eps)>\mu.$$
\end{remark}

\end{document}